\documentclass{article}
\usepackage{amsmath}

\newtheorem{theorem}{Theorem}
\newtheorem{proposition}{Proposition}

\newtheorem{lemma}{Lemma}
\newtheorem{remark}{Remark}
\newenvironment{proof}[1][Proof]{\noindent\textbf{#1.} }{\ \rule{0.5em}{0.5em}}
\input{tcilatex}

\begin{document}

\title{When is an ellipse inscribed in a quadrilateral tangent at the
midpoint of two or more sides ?}
\author{Alan Horwitz}
\date{3/25/17}
\maketitle

\begin{abstract}
In \cite{CL} it was shown that if there is a midpoint ellipse(an ellipse
inscribed in a quadrilateral, $Q$, which is tangent at the midpoints of all
four sides of $Q)$, then $Q$ must be a parallelogram. We strengthen this
result by showing that if $Q$\ is not a parallelogram, then there is no
ellipse inscribed in $Q$ which is tangent at the midpoint of three sides of $%
Q$. Second, the only quadrilaterals which have inscribed ellipses tangent at
the midpoint of even two sides of $Q$ are trapezoids or what we call a
midpoint diagonal quadrilateral(the intersection point of the diagonals of $%
Q $ coincides with the midpoint of at least one of the diagonals of $Q$).
\end{abstract}

\section{Introduction}

Among all ellipses inscribed in a triangle, $T$, the midpoint, or Steiner,
ellipse is interesting and well--known(see \cite{D}). It is the unique
ellipse tangent to $T$\ at the midpoints of all three sides of $T$ and is
also the unique ellipse of maximal area inscribed in $T$. What about
ellipses inscribed in quadrilaterals, $Q$ ? Not surprisingly, perhaps, there
is not always a midpoint ellipse--i.e., an ellipse inscribed in $Q$ which is
tangent at the midpoints of all four sides of $Q$; In fact, in \cite{CL} it
was shown that if there is a midpoint ellipse, then $Q$ must be a
parallelogram. That is, if $Q$\ is not a parallelogram, then there is no
ellipse inscribed in $Q$ which is tangent at the midpoint of all four sides
of $Q$; But can one do better than four sides of $Q$ ? In other words, if $Q$%
\ is not a parallelogram, is there an ellipse inscribed in $Q$ which is
tangent at the midpoint of three sides of $Q$ ? In Theorem \ref{T1} we prove
that the answer is no. In fact, unless $Q$ is a trapezoid(a quadrilateral
with at least one pair of parallel sides), or what we call a midpoint
diagonal quadrilateral(see the definition below), then there is not even an
ellipse inscribed in $Q$ which is tangent at the midpoint of two sides of $Q$%
(see Lemmas \ref{L1} and \ref{L2}) .

\begin{definition}
A convex quadrilateral, $Q$, is called a midpoint diagonal
quadrilateral(mdq) if the intersection point of the diagonals of $Q$
coincides with the midpoint of at least one of the diagonals of $Q$. 
\end{definition}

A parallelogram, p, is a special case of an mdq since the diagonals of p\
bisect one another. In \cite{H4} we discussed mdq's as a generalization of
parallelograms in a certain sense related to tangency chords and conjugate
diameters of inscribed ellipses. We note here that there are two types of
mdq's: Type 1, where the diagonals intersect at the midpoint of $D_{2}$ and
Type 2, where the diagonals intersect at the midpoint of $D_{1}$; 

What about uniqueness ? If $Q$ is an mdq, then the ellipse inscribed in $Q$
which is tangent at the midpoint of two sides of $Q$ is not unique. Indeed
we prove (Lemma \ref{L1}) that in that case there are two such ellipses.
However, if $Q$\ is a trapezoid, then the ellipse inscribed in $Q$ which is
tangent at the midpoint of two sides of $Q$ is unique (Lemma \ref{L2}).

Is there a connection with tangency at the midpoint of sides of $Q$ and the
ellipse of maximal area inscribed in $Q$ as with parallelograms ?

For trapezoids, we prove(Lemma \ref{L2}) that the unique ellipse of maximal
area inscribed in $Q$ is the unique ellipse tangent to $Q$ at the midpoint
of two sides of $Q$. However, for mdq's, the unique ellipse of maximal area
inscribed in $Q$ need not be tangent at the midpoint of any side of $Q$.

We use the notation $Q(A_{1},A_{2},A_{3},A_{4})$ to denote the quadrilateral
with vertices $A_{1},A_{2},A_{3}$, and $A_{4}$, starting with $A_{1}=$ lower
left corner and going clockwise. Denote the sides of $%
Q(A_{1},A_{2},A_{3},A_{4})$ by $S_{1},S_{2},S_{3}$, and $S_{4}$, going
clockwise and starting with the leftmost side, $S_{1}$; Given a convex
quadrilateral, $Q=Q(A_{1},A_{2},A_{3},A_{4})$, which is \textbf{not} a
parallelogram, it will simplify our work below to consider quadrilaterals
with a special set of vertices. In particular, there is an affine
transformation which sends $A_{1},A_{2}$, and $A_{4}$ to the points $%
(0,0),(0,1)$, and $(1,0)$, respectively. It then follows that $A_{3}=(s,t)$
for some $s,t>0$; Summarizing:

\begin{gather}
Q_{s,t}=Q(A_{1},A_{2},A_{3},A_{4}),  \label{Qst} \\
A_{1}=(0,0),A_{2}=(0,1),A_{3}=(s,t),A_{4}=(1,0)\text{.}
\end{gather}%
Since $Q_{s,t}$ is convex, $s+t>1$; Also, if $Q$ has a pair of parallel
vertical sides, first rotate counterclockwise by $90^{\circ }$, yielding a
quadrilateral with parallel horizontal sides. Since we are assuming that $Q$
is not a parallelogram, we may then also assume that $Q_{s,t}$ does not have
parallel vertical sides and thus $s\neq 1$. Note that any trapezoid which is
not a parallelogram may be mapped, by an affine transformation, to the
quadrilateral $Q_{s,1}$; Thus we may assume that $(s,t)\in G$, where 
\begin{equation}
G=\left\{ (s,t):s,t>0,s+t>1,s\neq 1\right\} \text{.}  \label{G}
\end{equation}%
The following result gives the points of tangency of any ellipse inscribed
in $Q_{s,t}$(see \cite{H2}, where some details were provided). We leave the
details of a proof to the reader. Those details will also be provided in a
forthcoming book \cite{H3}.

\begin{proposition}
\label{P1}(i) $E_{0}$ is an ellipse inscribed in $Q_{s,t}$\ if and only if
the general equation of $E_{0}$ is given by 
\begin{gather}
t^{2}x^{2}+{\large (}4q^{2}(t-1)t+2qt(s-t+2)-2st{\large )}xy+  \label{elleq}
\\
{\large (}(1-q)s+qt{\large )}^{2}y^{2}-2qt^{2}x-2qt{\large (}(1-q)s+qt%
{\large )}y+q^{2}t^{2}=0  \notag
\end{gather}%
for some $q\in J=(0,1)$. Furthermore, (\ref{elleq}) provides a one--to--one
correspondence between ellipses inscribed in $Q_{s,t}$ and points $q\in J$.

(ii) If $E_{0}$ is an ellipse given by (\ref{elleq}) for some $q\in J$, then 
$E_{0}$ is tangent to the four sides of $Q_{s,t}$\ at the points

$P_{1}=\left( 0,\dfrac{qt}{(t-s)q+s}\right) \in S_{1},$

$P_{2}=\left( \dfrac{(1-q)s^{2}}{(t-1)(s+t)q+s},\dfrac{t{\large (}s+q(t-1)%
{\large )}}{(t-1)(s+t)q+s}\right) \in S_{2},$

$P_{3}=\left( \dfrac{s+q(t-1)}{(s+t-2)q+1},\dfrac{(1-q)t}{(s+t-2)q+1}\right)
\in S_{3}$, and $P_{4}=(q,0)\in S_{4}$.
\end{proposition}

\begin{remark}
Using Proposition \ref{P1}, it is easy to show that one can always find an
ellipse inscribed in a quadrilateral, $Q$, which is tangent to $Q$\ at the
midpoint of at least \textbf{one} side of $Q$, and this can be done for any
given side of $Q$.
\end{remark}

For the rest of the paper we work with the quadrilateral $Q_{s,t}$ defined
above. The following lemma gives necessary and sufficient conditions for $%
Q_{s,t}$ to be an mdq.

\begin{lemma}
\label{Qstmdq}(i) $Q_{s,t}$\ is a type 1 midpoint diagonal quadrilateral if
and only if $s=t$.

(ii) $Q_{s,t}$\ is a type 2 midpoint diagonal quadrilateral if and only if\ $%
s+t=2$.
\end{lemma}

\begin{proof}
The diagonals of $Q_{s,t}$ are $D_{1}$: $y=\dfrac{t}{s}x$ and $D_{2}$: $%
y=1-x $, and they intersect at the point $P=\left( \dfrac{s}{s+t},\dfrac{t}{%
s+t}\right) $; The midpoints of $D_{1}$ and $D_{2}$\ are $M_{1}=\left( 
\dfrac{s}{2},\dfrac{t}{2}\right) $ and $M_{2}=\left( \dfrac{1}{2},\dfrac{1}{2%
}\right) $, respectively. Now $M_{2}=P\iff \dfrac{s}{s+t}=\dfrac{1}{2}$ and $%
\dfrac{t}{s+t}=\dfrac{1}{2}$, both of which hold if and only if $s=t$; That
proves (i); $M_{1}=P\iff \dfrac{s}{s+t}=\dfrac{1}{2}s$ and $\dfrac{t}{s+t}=%
\dfrac{1}{2}t$, both of which hold if and only if $s+t=2$. That proves (ii).
\end{proof}

The following lemma shows that affine transformations preserve the class of
mdq's. We leave the details of the proof to the reader.

\begin{lemma}
\label{Tmdq}Let $T:R^{2}\rightarrow R^{2}$ be an affine transformation and
let $Q$ be a midpoint diagonal quadrilateral. Then $Q^{\prime }=T(Q)$ is
also a midpoint diagonal quadrilateral.
\end{lemma}

The following result shows that among non--parallelograms, the only
quadrilaterals which have inscribed ellipses tangent at the midpoint of two
sides are the mdq's.

\begin{lemma}
\label{L1}Let $Q$\ be a convex quadrilateral in the $xy$ plane which is 
\textbf{not }a trapezoid.

(i) There is an ellipse inscribed in $Q$ which is tangent at the midpoint of
two sides of $Q$ if and only if $Q$ is a midpoint diagonal quadrilateral, in
which case there are two such ellipses.

(ii) There is no ellipse inscribed in $Q$ which is tangent at the midpoint
of three sides of $Q$.
\end{lemma}

\begin{proof}
By Lemma \ref{Tmdq} and standard properties of affine transformations, we
may assume that $Q=Q_{s,t}$, the quadrilateral given in (\ref{Qst}) with $%
(s,t)\in G$; The midpoints of the sides of $Q_{s,t}$ are given by $%
MP_{1}=\left( 0,\dfrac{1}{2}\right) \in S_{1}$, $MP_{2}=\left( \dfrac{s}{2},%
\dfrac{1+t}{2}\right) \in S_{2}$, $MP_{3}=\left( \dfrac{1+s}{2},\dfrac{t}{2}%
\right) \in S_{3}$, and $MP_{4}=\left( \dfrac{1}{2},0\right) \in S_{4}$; Now
let $E_{0}$ denote an ellipse inscribed in $Q_{s,t}$, and let $P_{j}\in
S_{j},j=1,2,3,4$ denote the points of tangency of $E_{0}$ with the sides of $%
Q_{s,t}$; By Proposition \ref{P1}(ii), 
\begin{gather}
P_{1}=MP_{1}\iff \dfrac{qt}{(t-s)q+s}=\dfrac{1}{2},  \label{32} \\
P_{2}=MP_{2}\iff \dfrac{(1-q)s}{(t-1)(s+t)q+s}=\dfrac{1}{2}\ \text{and}
\label{33a} \\
\dfrac{t{\large (}s+q(t-1){\large )}}{(t-1)(s+t)q+s}=\dfrac{1+t}{2}
\label{33b} \\
P_{3}=MP_{3}\iff \dfrac{s+q(t-1)}{(s+t-2)q+1}=\dfrac{1+s}{2}\text{ and}
\label{34a} \\
\dfrac{(1-q)t}{(s+t-2)q+1}=\dfrac{t}{2},  \label{34b} \\
P_{4}=MP_{4}\iff q=\dfrac{1}{2}\text{.}  \label{35}
\end{gather}%
Equations (\ref{32}) and (\ref{35}) each have the unique solutions $q_{1}=%
\dfrac{s}{s+t}$ and $q_{4}=\dfrac{1}{2}$, respectively. The system of
equations in (\ref{33a}) and (\ref{33b}) has unique solution $q_{2}=\dfrac{s%
}{t^{2}+st+s-t}$, and the system of equations in (\ref{34a}) and (\ref{34b})
has unique solution $q_{3}=\dfrac{1}{s+t}$; It is trivial that $%
q_{1},q_{3},q_{4}\in J=(0,1)$; Since $(s,t)\in G$, $t\left( s+t-1\right) >0$%
, which implies that $q_{2}\in J$. We now check which \textbf{pairs} of
midpoints of sides of $Q_{s,t}$ can be points of tangency of $E_{0}$; Note
that different values of $q$ yield distinct inscribed ellipses by the
one--to--one correspondence between ellipses inscribed in $Q_{s,t}$ and
points $q\in J$.

(a) $S_{1}$ and $S_{2}$: $q_{1}=q_{2}\iff \dfrac{s}{\allowbreak t\left(
s+t-1\right) +s}-\dfrac{s}{s+t}=$

$-\allowbreak \dfrac{st(s+t-2)}{\left( ts+t^{2}+s-t\right) \left( s+t\right) 
}=0\iff s+t=2$.

(b) $S_{1}$ and $S_{3}$: $q_{1}=q_{3}\iff \dfrac{1}{s+t}-\dfrac{s}{s+t}%
=\allowbreak -\dfrac{s-1}{s+t}=0$, which has no solution since $s\neq 1$.

(c) $S_{2}$ and $S_{3}$: $q_{2}=q_{3}\iff \dfrac{1}{s+t}-\dfrac{s}{%
\allowbreak t\left( s+t-1\right) +s}=\allowbreak \dfrac{(t-s)\left(
s+t-1\right) }{\left( ts+t^{2}+s-t\right) \left( s+t\right) }\allowbreak
=0\iff s=t$.

(d) $S_{1}$ and $S_{4}$: $q_{1}=q_{4}\iff \dfrac{1}{2}-\dfrac{s}{s+t}%
=\allowbreak -\dfrac{1}{2}\dfrac{s-t}{s+t}=0\iff s=t$.

(e) $S_{2}$ and $S_{4}$: $q_{2}=q_{4}\iff \dfrac{1}{2}-\dfrac{s}{\allowbreak
t\left( s+t-1\right) +s}=\allowbreak \dfrac{1}{2}\dfrac{(s+t)\left(
t-1\right) }{ts+t^{2}+s-t}=0$, which has no solution since $s+t\neq 0$ and $%
t\neq 1$.

(f) $S_{3}$ and $S_{4}$: $q_{3}=q_{4}\iff \dfrac{1}{2}-\dfrac{1}{s+t}%
=\allowbreak \dfrac{1}{2}\dfrac{s+t-2}{s+t}=0\iff s+t=2$.

That proves that there is an ellipse inscribed in $Q_{s,t}$ which is tangent
at the midpoints of $S_{1}$ and $S_{2}$ or at the midpoints of $S_{3}$ and $%
S_{4}$ if and only if $s+t=2$, and there is an ellipse inscribed in $Q_{s,t}$
which is tangent at the midpoints of $S_{2}$ and $S_{3}$ or at the midpoints
of $S_{1}$ and $S_{4}$ if and only if $s=t$; Furthermore, if $s\neq t$ or if 
$s+t\neq 2$, then there is no ellipse inscribed in $Q_{s,t}$ which is
tangent at the midpoint of two sides of $Q_{s,t}$. That proves (i) by Lemma %
\ref{Qstmdq}. To prove (ii), to have an ellipse inscribed in $Q_{s,t}$ which
is tangent at the midpoint of three sides of $Q_{s,t}$, those three sides
are either $S_{1}$, $S_{2}$, and $S_{3}$; $S_{1}$, $S_{2}$, and $S_{4}$; $%
S_{1}$, $S_{3}$, and $S_{4}$; or $S_{2}$, $S_{3}$, and $S_{4}$; By (a)--(f)
above, that is not possible.

For trapezoids inscribed in $Q$ we have the following result.
\end{proof}

\begin{lemma}
\label{L2}Assume that $Q$ is a trapezoid which is not a parallelogram. Then

(i) There is a unique ellipse inscribed in $Q$ which is tangent at the
midpoint of two sides of $Q$, and that ellipse is the unique ellipse of
maximal area inscribed in $Q$.

(ii) There is no ellipse inscribed in $Q$ which is tangent at the midpoint
of three sides of $Q$.
\end{lemma}

\begin{proof}
Again, by affine invariance, we may assume that $Q=Q_{s,1}$, the
quadrilateral given in (\ref{Qst}) with $t=1$; Note that $0<s\neq 1$; Now
let $E_{0}$ denote an ellipse inscribed in $Q_{s,1}$. Letting $MP_{j}\in
S_{j},j=1,2,3,4$ denote the corresponding midpoints of the sides and using
Proposition \ref{P1}(ii) again, with $t=1$, we have 
\begin{equation}
P_{1}=MP_{1}\iff \dfrac{q}{(1-s)q+s}=\dfrac{1}{2}\text{,}  \label{36}
\end{equation}%
\begin{equation}
P_{2}=MP_{2}\iff (1-q)s=\dfrac{s}{2}\text{,}  \label{37}
\end{equation}%
\begin{eqnarray}
P_{3} &=&MP_{3}\iff \dfrac{s}{(s-1)q+1}=\dfrac{1+s}{2}\text{ and}
\label{38a} \\
&&\dfrac{1-q}{(s-1)q+1}=\dfrac{1}{2}\text{,}  \label{38b}
\end{eqnarray}%
\begin{equation}
P_{4}=MP_{4}\iff q=\dfrac{1}{2}\text{.}  \label{39}
\end{equation}%
The unique solution of the equations in (\ref{37}) and in (\ref{39}) is $q=%
\dfrac{1}{2}\in J$; The unique solution of the equation in (\ref{36}) is $q=%
\dfrac{s}{1+s}\in J$, and the unique solution of the system of equations in (%
\ref{38a}) and (\ref{38b}) is $q=\dfrac{1}{1+s}\in J$; We now check which 
\textbf{pairs} of midpoints of sides of $Q_{s,1}$ can be points of tangency
of $E_{0}$:

(a) $q=\dfrac{1}{2}$ gives tangency at the midpoints of $S_{2}$ and $S_{4}$.

(b) $S_{1}$ and $S_{2}$ or $S_{1}$ and $S_{4}$: $\dfrac{s}{1+s}=\dfrac{1}{2}%
\iff s=1$.

(c) $S_{3}$ and $S_{2}$ or $S_{3}$ and $S_{4}$: $\dfrac{1}{1+s}=\dfrac{1}{2}%
\iff s=1$.

(d) $S_{1}$ and $S_{3}$: $\dfrac{s}{1+s}=\dfrac{1}{1+s}\iff s=1$.

Since we have assumed that $s\neq 1$, the only way to have an ellipse
inscribed in $Q_{s,1}$ which is tangent at the midpoint of two sides of $%
Q_{s,1}$ is if those sides are $S_{2}$ and $S_{4}$ and $q=\dfrac{1}{2}$.
That proves that there is a unique ellipse inscribed in $Q_{s,1}$ which is
tangent at the midpoint of two sides of $Q_{s,1}$. Now suppose that $E_{0}$
is any ellipse with equation $Ax^{2}+Bxy+Cy^{2}+Dx+Ey+F=0$, and let $a$ and $%
b$ denote the lengths of the semi--major and semi--minor axes, respectively,
of $E_{0}$. Using the results in \cite{S}, it can be shown that $a^{2}b^{2}=%
\dfrac{4\delta ^{2}}{\Delta ^{3}}$, where $\Delta =4AC-B^{2}$ and $\delta
=CD^{2}+AE^{2}-BDE-F\Delta $; By Proposition \ref{P1}(i), then, with $t=1$
and after some simplification, we have $a^{2}b^{2}=f(q)=\dfrac{s}{4}q\left(
1-q\right) $; $q=\dfrac{1}{2}$ clearly maximizes $f(q)$ and thus gives the
ellipse of maximal area inscribed in $Q_{s,1}$. That proves the rest of (i).
(ii) now follows easily and we omit the details.
\end{proof}

\begin{remark}
It can be shown (see \cite{H4}) that if $Q$\ is a trapezoid which is not a
parallelogram, then $Q$\ cannot be an mdq. Thus the only quadrilaterals
Lemmas \ref{L1} and \ref{L2} have in common are parallelograms.
\end{remark}

Since a convex quadrilateral which is not a parallelogram either has no two
sides which are parallel, or is a trapezoid, the following theorem follows
immediately from Lemma \ref{L1}(ii) and Lemma \ref{L2}(ii).

\begin{theorem}
\label{T1}Suppose that $Q$\ is a convex quadrilateral which is not a
parallelogram. Then there is no ellipse inscribed in $Q$ which is tangent at
the midpoint of three sides of $Q$.
\end{theorem}

\section{Examples\textbf{\ }}

(1) Let $Q$ be the quadrilateral with vertices, $(0,0),(0,1),(2,4)$, and $%
(1,1)$; It follows easily that $Q$ is a type 1 midpoint diagonal
quadrilateral. The ellipse with equation $10\left( x-\dfrac{2}{3}\right)
^{2}-10\left( x-\dfrac{2}{3}\right) \left( y-\dfrac{4}{3}\right) +4\left( y-%
\dfrac{4}{3}\right) ^{2}=\dfrac{5}{3}$ is tangent to $Q$ at $\left( 0,\dfrac{%
1}{2}\right) $ and at $\left( \dfrac{1}{2},\dfrac{1}{2}\right) $, the
midpoints of $S_{1}$ and $S_{4}$, respectively. The ellipse with equation $%
54\left( x-\dfrac{4}{5}\right) ^{2}-54\left( x-\dfrac{4}{5}\right) \left( y-%
\dfrac{8}{5}\right) +16\left( y-\dfrac{8}{5}\right) ^{2}=\dfrac{27}{5}$ is
tangent to $Q$ at$\allowbreak \left( 1,\dfrac{5}{2}\right) $ and at $%
\allowbreak \left( \dfrac{3}{2},\dfrac{5}{2}\right) $, the midpoints of $%
S_{2}$ and $S_{3}$, respectively. One can show that neither of these
ellipses is the ellipse of maximal area inscribed in $Q$. 

(2) Let $Q$ be the trapezoid with vertices $(0,0),(0,1),(2,1)$, and $(1,1)$;
The ellipse with equation

$\left( x-\dfrac{5}{4}\right) ^{2}-3\left( x-\dfrac{5}{4}\right) \left( y-%
\dfrac{1}{2}\right) +\dfrac{25}{4}\left( y-\dfrac{1}{2}\right) ^{2}=1$ is
tangent to $Q$ at $\left( 2,1\right) $ and at $\left( \dfrac{1}{2},0\right) $%
, the midpoints of $S_{2}$ and $S_{4}$, respectively.

\end{document}